\documentclass[10pt,twoside]{article}

\usepackage{amssymb, amsthm, amsbsy, amsmath, latexsym, mathrsfs}
\usepackage[hidelinks]{hyperref}
\newcommand{\lr}[1]{\lbrace #1 \rbrace}

\newtheorem{defi}{Definition}
\newtheorem{thm}{Theorem}

\newtheorem{ex}{Example}

\usepackage[T1]{fontenc}
\usepackage{color}

\begin{document}

\thispagestyle{empty}  
\setlength{\oddsidemargin}{1cm} 
\setlength{\evensidemargin}{1cm}
\setlength{\topmargin}{2cm}

\begin{center}
{ \Large {\bf {Association schemes on triples over few vertices   }}}\\
\vspace{.2in}
{\sc Jose Maria P. Balmaceda}\\

{\small
Institute of Mathematics\\
University of the Philippines Diliman, 1101 Diliman, Quezon City, Philippines\\

jpbalmaceda@up.edu.ph\\
\bigskip

{\sc Dom Vito A. Briones}\\

{\small
Institute of Mathematics\\
University of the Philippines Diliman, 1101 Diliman, Quezon City, Philippines\\
dabriones@up.edu.ph}\\
}

\end{center}

\medskip

\begin{abstract}

In this paper, we obtain classification results for higher-dimensional analogues of classical association schemes called association schemes on triples (ASTs). We present an algorithm that enumerates all ASTs on a fixed number of vertices whose nontrivial relations are invariant under the action of some group. Applying this algorithm to three, four, and five vertices along with appropriate group actions yields the unique AST over three vertices, the unique symmetric ASTs over four or five vertices, the unique AST over four vertices with two nontrivial relations, and the unique nontrivial circulant AST over five vertices.

\medskip
\noindent
{\bf Keywords:} association scheme on triples, algebraic combinatorics, classification

\noindent
{\bf MSC 2020:} 05E30
\end{abstract}
\bigskip

\section{Introduction}
Initially introduced by Bose and Shimamoto for designing statistical experiments \cite{Bose1952,bannai_algebraic_1984}, classical association schemes have since found widespread applications in various algebraic-combinatorial structures. By unifying seemingly different mathematical objects under a single framework, such applications include permutation groups \cite{wielandt_finite_2015,cameron_permutation_1999,higman_coherent_1975,evdokimov_permutation_2009}, discrete geometries, graphs and directed graphs \cite{higman_coherent_1975,bannai_current_1986,bannai_algebraic_1984,godsil_algebraic_1993,godsil_algebraic_2001,brouwer_distance-regular_1989,efimov_distance-regular_2020,bang_geometric_2014,haemers_strongly_2010,vieira_generalized_2015,delsarte_association_1998,bannai_current_1986}, coding theory and designs \cite{delsarte_association_1998,bailey2004association}, certain types of algebras \cite{ponomarenko_preface_2009,hanaki_terwilliger_2011,blau_table_2009,xu_structure_2011,muzychuk_schur_2009}, and generalizations of groups and their representations \cite{zieschang_trends_2009,zieschang1995homogeneous,zieschang2006algebraic,zieschang_theory_2005,rassy_basic_1998}.

In its relation form, a classical association scheme of order $m$ on a nonempty set $\Omega$ is a partition $X=\lr{R_i}_{i=0}^m$ of $\Omega \times \Omega$ that satisfies certain symmetry requirements. These requirements are robust and flexible enough to equip classical association schemes with algebraic and combinatorial properties capable of describing the variety of mathematical structures mentioned prior. For instance, the linear span of the adjacency matrices of these relations is a Jacobson semisimple associative binary algebra that satisfies certain duality properties \cite{bannai_algebraic_1984}. Restrictions on the parameters of the adjacency algebra and its dual algebra can then be used for the classification theorems of certain graphs \cite{bannai_algebraic_1984}. With these applications of classical association schemes in mind, Mesner and Bhattacharya created a higher-dimensional analogue of association schemes called association schemes on triples or ASTs \cite{mesner_association_1990}. In ASTs, the relations and resulting adjacency algebras are ternary instead of binary. In the same paper, some basic properties of ASTs were observed; in particular, close relationships with ternary algebras, 2-designs, two-graphs, and two-transitive permutation groups were discerned. However, the study of ASTs is still in its infancy, as many structural properties of ASTs are still unknown. For example, it is not yet known if there is an analogue for the desirable semisimplicity and duality properties of classical association schemes in ASTs. However, \cite{mesner_ternary_1994} introduced ``identity pairs'' and ``inverse pairs'' as analogues of the usual multiplicative identity and multiplicative inverses in binary algebras in hopes of providing tools that may someday illuminate the ternary algebra structure of ASTs. In addition, \cite{Zealand2021} introduced another family of ASTs called circulant ASTs whose nontrivial relations are invariant under a transitive cyclic subgroup of the symmetric group. 

Motivated to aid in the study of ASTs by classifying examples on a small number of vertices, we mirror the classification of classical association schemes with few vertices up to isomorphism. Such classifications of classical association schemes are found in works such as \cite{Nomiyama1995,HIRASAKI1997,Hirasaka1996,Miyamoto1998,Hanaki1998,HANAKI2000, hanaki1999classification,See1998,Hanaki2003}. We approach the classification problem for ASTs by providing an algorithm which outputs the ASTs $X=\lr{R_i}_{i=0}^m$ of order $m$ over $n$ vertices such that the nontrivial relations of the ASTs are invariant under the action of some group. Applying this algorithm to $n\in \lr{3,4,5}$ and the appropriate groups, we obtain the unique AST over three vertices, the unique symmetric ASTs over $n\in \lr{4,5}$ vertices, the unique AST over four vertices with two nontrivial relations, and the unique nontrivial circulant AST over five vertices.

%Classification work was done for example in ...,..., and ..., wherein they determined the number of isomorphism classes of ASTs and their algebras 

%In 1990, Mesner and Bhattacharya introduced association schemes on triples.

%We provide an algorithm for this. Applying this to whatever yields something.

\section{Preliminaries}

In this section, we introduce association schemes on triples which are higher-dimensional analogues of classical association schemes. Due to their relevance in our examples, we also define symmetric ASTs, circulant ASTs, and isomorphism of ASTs. Many of the definitions, notations, and theorems we use are from \cite{mesner_association_1990} and \cite{Zealand2021}. 

We begin by defining order $m$ association schemes on triples over $n$ vertices.

\begin{defi}\label{def_ast}
Let $\Omega$ be a finite set of cardinality $n\geq 3$. An association scheme on triples (AST) of order $m\geq4$ on $\Omega$ is a partition $X=\lr{R_i}_{i=0}^m$ of $\Omega \times \Omega \times \Omega$ that satisfies the following properties.

\begin{enumerate}
    \item For any $i\in \lr{0,\ldots,m}$ there exists a constant $n_i^{(3)}\in \mathbb{N} \cup \lr{0}$ such that \[\vert \lr{z\in \Omega : (x,y,z)\in R_i} \vert=n_i^{(3)},\] for any pair of distinct elements $x,y\in \Omega$.
    \item For any $i,j,k,l \in \lr{0,\ldots,m}$, there exists a constant $p_{ijk}^l \in \mathbb{N} \cup \lr{0}$ such that \[\vert \lr{w:(w,y,z)\in R_i,\; (x,w,z)\in R_j,\; \text{and}\; (x,y,w)\in R_k} \vert = p_{ijk}^l ,\] for any $(x,y,z)\in R_l$.
    \item For any $i\in \lr{0,\ldots,m}$ and $\sigma\in S_3$, there exists a $j\in \lr{0,\ldots,m}$ such that \[R_i^{\sigma}:=\lr{(x_{\sigma(1)},x_{\sigma(2)},x_{\sigma(3)}):(x_1,x_2,x_3)\in R_i} = R_j.\]
    \item The first four relations are given by the following.
    \begin{align*}
        R_0 &= \lr{(x,x,x):x\in \Omega},\\
        R_1 &= \lr{(y,x,x):x,y\in \Omega,\; x\neq y},\\
        R_2 &= \lr{(x,y,x):x,y\in \Omega,\; x\neq y},\\
        R_3 &= \lr{(x,x,y):x,y\in \Omega,\; x\neq y}.
    \end{align*}
\end{enumerate}
\end{defi}
The relations $R_0,\ldots,R_3$ are the trivial relations and the remaining relations are the nontrivial relations. %If $p_{ijk}^l=p_{i'j'k'}^l$ for all permutations $(i',j',k')$ of $(i,j,k)$, we say that the AST $X$ is commutative. 
For example, analogous to how transitive groups yield classical association schemes, two-transitive groups yield ASTs \cite{mesner_association_1990}. Indeed, if $G$ is a two-transitive group acting on a set $\Omega$ of $n\geq 4$ elements, the orbits of $G$ in its natural action on $\Omega \times \Omega \times \Omega$ form an AST.

Paralleling the definition of symmetric classical association schemes, an AST $X=\lr{R_i}_{i=0}^m$ is symmetric when $R_i^\sigma=R_i$ for all nontrivial relations $R_i\in X$ and all $\sigma \in S_3$. In other words, each nontrivial relation is invariant under the action of $S_3$ by coordinate permutation. The intersection numbers $p_{ijk}^l$ of these ASTs were studied in \cite{mesner_association_1990}, along with some relationships of symmetric ASTs with partial 3-designs. It was also shown in the same paper that any nontrivial relation of an AST that is invariant under coordinate permutation yields a family of 2-designs. Moreover, they showed that 2-designs yield order $5$ ASTs and that a converse exists for this construction. 

If we instead consider ASTs on $\Omega$ where all nontrivial relations are invariant under a common transitive cyclic subgroup of $S_\Omega$, we obtain circulant ASTs \cite{Zealand2021}; that is, an AST $X=\lr{R_i}_{i=0}^m$ is circulant if there exists a transitive cyclic subgroup $G\leq S_\Omega$ such that for any $i\in \lr{0,\ldots,m}$, we have \[\lr{(gx,gy,gz):g\in G, \; (x,y,z)\in R_i}\subseteq R_i.\] Such ASTs were introduced and studied in \cite{Zealand2021}, where it was shown that circulant ASTs arise in correspondence with so-called AST-regular partitions of a certain subset of $\Omega \times \Omega$. In fact, they showed that in circulant ASTs, each nontrivial relation is the disjoint union of so-called ``thin'' circulant relations. For example, the ASTs obtained from the two-transitive action of $AGL(1,p)$ on the Galois Field $GF(p)$, where $p$ is prime, is circulant. This is due to the subgroup of translations in $AGL(1,p)$, which is cyclic and which acts transitively on $GF(p)$.

To close this section, we provide an analogue of combinatorial isomorphisms for classical association schemes \cite{Nomiyama1995} to association schemes on triples. To do so, let \[\sigma(R)=\lr{(\sigma(x),\sigma(y),\sigma(z)):(x,y,z)\in R}\] for each $R\subseteq \Omega \times \Omega \times \Omega$ and $\sigma \in S_\Omega$.

\begin{defi}\label{def_iso}
Let $X=\lr{R_i}_{i=0}^m$ and $Y=\lr{S_i}_{i=0}^m$ be two order $m$ ASTs over a set $\Omega$ of $n\geq 3$ vertices.  We say that $X$ and $Y$ are isomorphic ASTs if there exists $\sigma \in S_\Omega$ such that \[\lr{\sigma(R_i)}_{i=0}^m=\lr{S_i}_{i=0}^m,\] where the equality is on unordered sets.
\end{defi}

In other words, two ASTs $X$ and $Y$ on $\Omega$ are isomorphic if relabelling the elements of $\Omega$ using $\sigma$ and then rearranging the $\sigma(R_i)$ yield the relations $\lr{S_i}_{i=0}^m$ of $Y$.
For computations in the algorithm later on, we restate this isomorphism condition in terms of a group action by the symmetric group $S_\Omega$ on the set of partitions of $\Omega\times \Omega \times \Omega$. Indeed, $S_\Omega$ acts on the set of all unordered partitions of $\Omega \times \Omega \times \Omega$ by 
\begin{equation}\tag{$\dagger$}\label{eq_iso}
  \sigma (\lr{R_i}_{i=0}^m) = \lr{\sigma(R_i)}_{i=0}^m .
\end{equation} It follows that two ASTs $X$ and $Y$ are isomorphic if and only if $X$ and $Y$ belong to the same orbit under this action of $S_\Omega$.

%For example, let 

\section{Algorithm for determining the number of isomorphism classes of ASTs}

In this section, we present an algorithm for determining the number of isomorphism classes of ASTs on a fixed number of vertices such that each relation is invariant under some group action on $\Omega \times \Omega \times \Omega$. We provide the algorithm below and prove its validity. 

\begin{thm}\label{thm_algo}
Let $m\geq 4$, $\Omega$ be a set of size $n\geq 3$, and $G$ be a group acting on $\Omega \times \Omega \times \Omega$.  

\begin{enumerate}
    \item Let $\lr{R_i}_{i=0}^3$ be the trivial relations, $\mathbb{O}$ be the set of orbits of $G$ on $\Omega \times \Omega \times \Omega\setminus \bigcup_{i=0}^3{R_i}$, and $\mathbb{P}$ be the set of size $m-3$ (unordered) partitions of $\mathbb{O}$.
    \item Let $P \in \mathbb{P}$ and $U$ be an element of $P$. Define $A_U=\lr{\bigcup_{O \in U} O }$. Define $\hat{P}$ to be the partition of $\Omega\times \Omega \times \Omega$ given by \[\hat{P}=\lr{R_0,R_1,R_2,R_3}\cup \lr{A_U:U\in P}.\] Further, let $\hat{\mathbb{P}}=\lr{\hat{P}:P\in\mathbb{P}}.$
    \item By checking conditions 1, 2, and 3 of Definition \ref{def_ast}, let $\mathbb{A}\subseteq \hat{\mathbb{P}}$ be the members of $\hat{\mathbb{P}}$ that are ASTs.
    \item Using the action of $S_\Omega$ described in \eqref{eq_iso}, determine the ASTs in $\mathbb{A}$ up to isomorphism by computing which ones belong to the same orbit. 
    \item Let $\mathbb{T}\subseteq \mathbb{A}$ be a transversal of the orbits in the previous step.
\end{enumerate}
Then $\mathbb{T}$ gives all ASTs of order $m$ on $\Omega$ (up to isomorphism) such that each nontrivial relation is invariant under the action of $G$.
\end{thm}

\begin{proof}
Any AST of order $m$ is a size $m+1$ partition of $\Omega\times \Omega \times \Omega$ whose first four relations are the trivial relations $R_0,R_1,R_2$, and $R_3$. Since the trivial relations are always included in any AST, the problem reduces to determining the size $m-3$ partitions of $\Omega \times \Omega \times \Omega \setminus \bigcup_{i=0}^3{R_i}$ whose members are invariant under $G$ and which serve as the nontrivial relations of some AST on $\Omega$. In order to find such partitions, we note that a ternary relation $R\subseteq \Omega \times \Omega \times \Omega$ is invariant under $G$ if and only if $
R$ is a union of orbits of $G$. 

The first step in the algorithm is to obtain the set $\mathbb{O}$ of orbits of $G$ and find the set $\mathbb{P}$ of all its size $m-3$ partitions. The second step in the algorithm is to establish a search space for partitions of $\Omega \times \Omega \times \Omega$ that contains all ASTs of size $m$ where each nontrivial relation is invariant under $G$. The set $\hat{\mathbb{P}}$ is our search space, as we have observed that any AST with the desired properties will have each of its nontrivial relations as unions of members of $\mathbb{O}$. The third step determines the subset $\mathbb{A}$ of ${\hat{\mathbb{P}}}$ that consists of the partitions of $\Omega\times \Omega \times \Omega$ that are ASTs with the desired properties. The fourth step then provides how many of the desired ASTs exist, up to isomorphism. Finally, by taking a representative from each isomorphism class, the fifth step yields all ASTs (up to isomorphism) of order $m$ on $\Omega$ such that each nontrivial relation is invariant under the action of $G$.
\end{proof}

\section{Applications to \texorpdfstring{ ${n=3,4,5}$}{n=3,4,5}}

In this section, we apply the results of Theorem \ref{thm_algo} on $n\in\lr{3,4,5}$ and suitable choices of $G$. This yields the unique AST over three vertices, the unique symmetric ASTs over $n\in \lr{4,5}$ vertices, the unique AST over four vertices with two nontrivial relations, and the unique nontrivial circulant AST over five vertices. The data was obtained through source code written for and executed by the computer algebra system GAP 4.11.1. The source code is located in the appendix.

The first example uses the action of the trivial group to yield the unique AST over three vertices. Necessarily, this AST is symmetric and circulant.
\begin{ex}\label{ex_trivial}
Let $\Omega=\lr{1,2,3}$, $G$ be the trivial group, and $m-3$ range over $\lr{1,\ldots,6=|\mathbb{O}|}$. Any ternary relation is invariant under the trivial group; hence, ranging over all possible partition sizes yields all ASTs over three vertices. There is only one such AST and its sole nontrivial relation is $\Omega \times \Omega \times \Omega \setminus\bigcup_{i=0}^3{R_i}$. \end{ex}

The second example uses the orbits of $S_3$ acting on $\Omega \times \Omega \times \Omega$ by coordinate permutation as building blocks for the nontrivial relations of ASTs. This yields the unique symmetric ASTs over $n\in \lr{4,5}$ vertices, each of which is also necessarily circulant.

\begin{ex}
Let $\Omega=\lr{1,\ldots,n}$ with $n\in \lr{4,5}$, $G=S_3$, and $m-3$ range over $\lr{1,\ldots,|\mathbb{O}|}$. The ternary relations invariant under coordinate permutation are those who are unions of orbits of $S_3$; hence, ranging over all possible partition sizes yields all symmetric ASTs over four or five vertices. There is only one such AST for $n=4$ and only one such AST for $n=5$. In each case, the AST has a lone nontrivial relation given by $\Omega \times \Omega \times \Omega \setminus\bigcup_{i=0}^3{R_i}$.
\end{ex}

The following example shows that the symmetric AST over four vertices obtained from the prior example is also the only circulant AST over four vertices.

\begin{ex}\label{ex_circ}
Let $\Omega=\lr{1,2,3,4}$, $G$ be the cyclic group generated by $(1,2,3,4)\in S_4$, and $m-3$ range over $\lr{1,\ldots,|\mathbb{O}|}$. Any ternary relation invariant under $G$ will be unions of orbits of $G$ in $\Omega \times \Omega \times \Omega$; hence, ranging over all possible partition sizes yields all circulant ASTs (with respect to $G$) over four vertices. In fact, if $H$ were any transitive cyclic subgroup of $S_4$ and $Y$ were any circulant AST with respect to $H$, then we may relabel the elements of $\Omega$ to obtain an isomorphic AST that is circulant with respect to $G$. Hence, our algorithm actually produces all circulant ASTs over four vertices, up to isomorphism. There is only one such AST and it has a lone nontrivial relation $\Omega \times \Omega \times \Omega \setminus\bigcup_{i=0}^3{R_i}$. 
\end{ex}

The next example again uses the trivial group but fixes $m$ to be 5. This yields the unique size 6 AST on four vertices. 

\begin{ex}
Let $\Omega=\lr{1,2,3,4}$, $G$ be the trivial group, and $m=5$. Reasoning as in Example \ref{ex_trivial}, the algorithm produces all order 5 ASTs over four vertices. We find that there is only one such AST. Its nontrivial relations $R_4$ and $R_5$ are given by the following.
\begin{align*}
    R_4 =& \Big\{ (  1, 2, 3 ), ( 1, 3, 4 ), ( 1, 4, 2 ), ( 2, 1, 4 ), ( 2, 3, 1 ), ( 2, 4, 3 ), \\ 
    & ( 3, 1, 2 ), ( 3, 2, 4 ),
      ( 3, 4, 1 ), ( 4, 1, 3 ), ( 4, 2, 1 ), ( 4, 3, 2 ) \Big\},\\
    R_5 =& 
  \Big\{ ( 1, 2, 4 ), ( 1, 3, 2 ), ( 1, 4, 3 ), ( 2, 1, 3 ), ( 2, 3, 4 ), ( 2, 4, 1 ), \\
  & ( 3, 1, 4 ),( 3, 2, 1 ),
      ( 3, 4, 2 ), ( 4, 1, 2 ), ( 4, 2, 3 ), ( 4, 3, 1 )\Big\}.
\end{align*}

 In fact, \cite{mesner_association_1990} says that the two-transitive group PSL(2,3) acting on the projective line $PG(1,3) $ yields an AST of size 6. In the same paper, an AST of size 6 is also obtained from the action of the two-transitive group $AGL(1,4)$ on the Galois field $GF(4)$. Hence, the unique AST of order 5 over four vertices is the AST induced by the action of $PSL(2,3)$ on $PG(1,3)$ and is isomorphic to the AST obtained from the action of $AGL(1,4)$ on $GF(4)$. 
\end{ex}

The last example uses the transitive action of the cyclic subgroup generated by the permutation $(1,2,3,4,5)$ to obtain two circulant ASTs over five vertices.

\begin{ex}
Let $\Omega=\lr{1,2,3,4,5}$, $G$ be the cyclic group generated by $(1,2,3,4,5)\in S_5$, and $m-3$ range over $\lr{1,\ldots,|\mathbb{O}|}$. Reasoning as in Example \ref{ex_circ}, our algorithm produces all circulant ASTs over five vertices. There are two such ASTs. The first is the AST whose sole nontrivial relation is $\Omega \times \Omega \times \Omega \setminus\bigcup_{i=0}^3{R_i}$, as obtained before. The remaining AST has size 7, with nontrivial relations $R_4, R_5$, and $R_6$ as given below.

\begin{align*}
    R_4 =& \Big\{ ( 1, 2, 3 ), ( 1, 3, 5 ), ( 1, 4, 2 ), ( 1, 5, 4 ), ( 2, 1, 5 ), ( 2, 3, 4 ), ( 2, 4, 1 ), ( 2, 5, 3 ),
  ( 3, 1, 4 ), ( 3, 2, 1 ), \\
  &( 3, 4, 5 ), ( 3, 5, 2 ), ( 4, 1, 3 ), ( 4, 2, 5 ), ( 4, 3, 2 ), ( 4, 5, 1 ),
  ( 5, 1, 2 ), ( 5, 2, 4 ), ( 5, 3, 1 ), ( 5, 4, 3 ) \Big\},\\
  R_5=& \Big\{ ( 1, 2, 4 ), ( 1, 3, 2 ), ( 1, 4, 5 ), ( 1, 5, 3 ), ( 2, 1, 4 ), ( 2, 3, 5 ), ( 2, 4, 3 ), ( 2, 5, 1 ),
  ( 3, 1, 2 ), ( 3, 2, 5 ), \\ & ( 3, 4, 1 ), ( 3, 5, 4 ), ( 4, 1, 5 ), ( 4, 2, 3 ), ( 4, 3, 1 ), ( 4, 5, 2 ),
  ( 5, 1, 3 ), ( 5, 2, 1 ), ( 5, 3, 4 ), ( 5, 4, 2 ) \Big\},\\
  R_6=&\Big\{ ( 1, 2, 5 ), ( 1, 3, 4 ), ( 1, 4, 3 ), ( 1, 5, 2 ), ( 2, 1, 3 ), ( 2, 3, 1 ), ( 2, 4, 5 ), ( 2, 5, 4 ),
  ( 3, 1, 5 ), ( 3, 2, 4 ), \\&( 3, 4, 2 ), ( 3, 5, 1 ), ( 4, 1, 2 ), ( 4, 2, 1 ), ( 4, 3, 5 ), ( 4, 5, 3 ),
  ( 5, 1, 4 ), ( 5, 2, 3 ), ( 5, 3, 2 ), ( 5, 4, 1 ) \Big\}.
\end{align*}

In fact, \cite{mesner_association_1990} says that the two-transitive action of the group $AGL(1,5)$ on the Galois field $GF(5)$ yields an AST of size 7. Since this AST is circulant, it follows that this must be the size 7 AST obtained above.
\end{ex}

\appendix
\section{Source codes}
This appendix provides the source codes for GAP 4.11.1 utilized for obtaining the classification results in the article.

\begin{enumerate}
\item This is the program that executes the algorithm provided in Theorem \ref{thm_algo}. The size of $\Omega$, the desired group action by $G$, and the order $m$ of the desired ASTs are set manually, as guided by the comments therein. The program is dependent upon the auxiliary programs in items 2, 3, 4, and 5 of this list. 

\begin{verbatim}
n:= #set the desired set size here#; 
set:=[1..n];
cart:=Cartesian(set,set,set); #yields cartesian product

#setting the trivial relations
R:=[];
R[1]:=[];
    for c in cart do
        if c[1]=c[2] and c[1]=c[3] then Append(R[1],[c]); fi;
    od;
R[1]:=Set(R[1]);
R[2]:=[];
    for c in cart do
        if not(c[1]=c[2]) and c[2]=c[3] then Append(R[2],[c]); fi;
    od;	
R[2]:=Set(R[2]);	
R[3]:=[];
    for c in cart do
        if not(c[1]=c[2]) and c[1]=c[3] then Append(R[3],[c]); fi;
    od;
R[3]:=Set(R[3]);
R[4]:=[];
    for c in cart do 
        if not(c[1]=c[3]) and c[1]=c[2] then Append(R[4],[c]); fi;
    od;
R[4]:=Set(R[4]);

for i in[1..4] do 
    SubtractSet(cart,R[i]);
od;	

#determining the orbits of the desired group action
symblocks:=#set desired action here#;
#e.g. Orbits(SymmetricGroup(3),cart,Permuted);
symblocks2:=[];
    for i in [1..Size(symblocks)] do 
        symblocks2[i]:=Set(symblocks[i]);
    od;
symblocks:=symblocks2;

#determining partitions
part:=PartitionsSet(symblocks,#set partition size here#);;

#determining orbits of partitions
distinct:=Orbits(SymmetricGroup(n),part,partact);;

#obtain a representative for each partition orbit
transversal:=[];
    for i in [1..Size(distinct)] do 
        Append(transversal,[distinct[i][1]]);
    od;

#obtain candidates for ASTs
unioned:=[];
    for i in [1..Size(transversal)] do
        unioned[i]:=[];
        for j in [1..Size(transversal[i])] do
            unioned[i][j]:=[];
            unioned[i][j][1]:=[];
            for k in [1..Size(transversal[i][j])] do
                unioned[i][j][1]:=UnionSet(unioned[i][j][1],
                transversal[i][j][k]); 
            od;
    od; od;

S:=[];
    for i in [1..Size(unioned)] do 
        S[i]:=ShallowCopy(R);
        for j in [1..Size(unioned[i])] do 
            Append(S[i],[unioned[i][j][1]]);
        od;
    od;

#Check for AST axioms
for i in [1..Size(S)] do 
    if
        (valencycheck(S[i],[1..n]))=true and 
        (regularitycheck(S[i],[1..n]))=true and
        (permclosedcheck(S[i],[1..n]))=true then
        Print(i); Print("\n"); fi;
od;    \end{verbatim}

\item This is an auxiliary program for determining the orbit of a partition of $\Omega \times \Omega \times \Omega$ under the action given by \eqref{eq_iso}.

\begin{verbatim}partact:= function(P,g) #P for a partition g for Sym elt 

local size, image, i,j;
size:=Size(P);
image:=[];

for i in [1..size] do
    image[i]:=[]; 
        for j in [1..Size(P[i])] do 
            image[i][j]:=Set(OnSetsTuples(P[i][j],g));
        od; 
image[i]:=Set(image[i]); od;

return Set(image);

end;;\end{verbatim}

    \item This is an auxiliary program for determining whether or not the first condition of Definition \ref{def_ast} is satisfied.
    \begin{verbatim}valencycheck:=function(R,S)
	
local truth, rel,countinit,s,x,y,z,countproper;
truth:=true;
	
for rel in R do 
    countinit:=0;
	
    for s in S do
        if ([S[1],S[2],s] in rel) then countinit:=countinit+1; fi;
    od;
    for x in S do 
    for y in S do
        countproper:=0;
        if not(x=y) then 
            for z in S do
                if [x,y,z] in rel then countproper:=countproper+1; fi;
            od;
            if not(countproper=countinit) then return false; fi;
        fi;
    od; od; 
od;
	
return truth;
	
end;\end{verbatim}

\item This is an auxiliary program for determining whether or not the second condition of Definition \ref{def_ast} is satisfied.

\begin{verbatim}regularitycheck := function(R,S)
	
local truth, index, i,j,k,l, a, b, countinit, countproper, z;
truth:=true;
index:=[1..Size(R)];
	
for i in index do 
for j in index do
for k in index do 
for l in index do 
    countinit:=0;
    a:=R[l][1];
        for z in S do 
            if ([z,a[2],a[3]] in R[i]) and ([a[1],z,a[3]] in R[j]) and
            ([a[1],a[2],z] in R[k]) then countinit:=countinit+1; fi;
        od;
    for b in R[l] do 
    countproper:=0;
        for z in S do 
           if ([z,b[2],b[3]] in R[i]) and ([b[1],z,b[3]] in R[j]) and
           ([b[1],b[2],z] in R[k]) then countproper:=countproper+1;fi;
        od;
    if not(countproper =countinit) then 
    return false; fi;
    od;
od;	od; od; od; 
	
return truth;
	
end;;   \end{verbatim}

\item This is an auxiliary program for determining whether or not the third condition of Definition \ref{def_ast} is satisfied.
\begin{verbatim}permclosedcheck:= function(R,S)
	
local sym, truth, s, rel, relcopy, i;
sym:=SymmetricGroup(3);
truth:=true;

for s in sym do 
    for rel in R do 
        relcopy:=[];
        for i in [1..Size(rel)] do
            relcopy[i]:=Permuted(rel[i],s);
        od;	
        if not(Set(relcopy) in R) then return false; fi;
    od;
od;
	
return truth;
	
end;;   \end{verbatim}

\end{enumerate}

%\printbibliography
\bibliographystyle{amsplain}
\bibliography{AST}

\end{document}